\documentclass[12pt]{amsart}
\usepackage[utf8]{inputenc}
\usepackage[T1]{fontenc}
\usepackage[all]{xy}
\usepackage{lipsum}
\usepackage{url}
\usepackage{tikz}
\usepackage{stackrel}
\usepackage{color}

\usepackage{amsmath,amsthm,amssymb}

\usepackage{xcolor}
\usepackage{graphicx}

\newtheorem{theorem}{Theorem}[section]
\newtheorem{lemma}[theorem]{Lemma}
\newtheorem{example}[theorem]{Example}
\newtheorem{proposition}[theorem]{Proposition}
\theoremstyle{definition}
\newtheorem{definition}[theorem]{Definition}

\newtheorem{remark}[theorem]{Remark}
\newtheorem{corollary}[theorem]{Corollary}

\numberwithin{equation}{section}

\begin{document}


\renewcommand{\bf}{\bfseries}
\renewcommand{\sc}{\scshape}

\title[Borsuk-Ulam property and Sectional Category]%
{Borsuk-Ulam property and Sectional Category}

\author{Cesar A. Ipanaque Zapata}
\address{Departamento de Matem\'atica, IME Universidade de S\~ao Paulo\\
Rua do Mat\~ao 1010 CEP: 05508-090 S\~ao Paulo-SP, Brazil}

\email{cesarzapata@usp.br}
\email{dlgoncal@ime.usp.br}

\author{Daciberg L. Gonçalves} 


\thanks {The first author would like to thank grant\#2016/18714-8 and grant\#2022/03270-8, S\~{a}o Paulo Research Foundation (FAPESP) for financial support. The second  author was partially supported by the FAPESP “Projeto Tem\'atico-FAPESP Topologia Alg\'ebrica, Geom\'etrica e Diferencial” 2016/24707-4 (São Paulo-Brazil).}

\begin{abstract} For a Hausdorff space $X$, a free involution $\tau:X\to X$ and a Hausdorff space $Y$, we discover a connection between the sectional category of the double covers $q:X\to X/\tau$ and $q^Y:F(Y,2)\to D(Y,2)$ from the ordered configuration space $F(Y,2)$ to its unordered quotient $D(Y,2)=F(Y,2)/\Sigma_2$, and the Borsuk-Ulam property (BUP) for the triple $\left((X,\tau);Y\right)$. Explicitly, we demonstrate that the triple $\left((X,\tau);Y\right)$ satisfies the BUP if the sectional category of $q$ is bigger than the sectional category of $q^Y$. This property connects a standard problem in Borsuk-Ulam theory to current research trends in sectional category. As an application of our results, we show that the index of $(X,\tau)$ coincides with the sectional category of the quotient map $q:X\to X/\tau$ minus 1 for any paracompact space $X$. In addition, we present some new results relating  Borsuk-Ulam theory and sectional category. 
\end{abstract}
\maketitle

\section{Introduction}\label{secintro}

\noindent Let $\left((X,\tau);Y\right)$ be a triple where $X$ is a Hausdorff space, $\tau:X\to X$ is a fixed-point free involution and $Y$ is a Hausdorff space. We say that $\left((X,\tau);Y\right)$ \textit{satisfies the Borsuk-Ulam property} (which we shall routinely abbreviate to BUP) if  for every continuous map $f:X\to Y$ there exists a point $x\in X$ such that $f(\tau(x))=f(x)$. 

Let $S^m$ be the $m$-dimensional sphere, $A:S^m\to S^m$ the antipodal involution (i.e., $A(x)=-x$ for any $x\in S^m$) and $\mathbb{R}^n$ the $n$-dimensional Euclidean space. The famous Borsuk-Ulam theorem \cite{borsuk1933} states that for every continuous map $f:S^m\to \mathbb{R}^m$ there exists a point $x\in S^m$ such that $f(-x)=f(x)$, i.e., the triple $\left((S^m,A);\mathbb{R}^m\right)$ satisfies the Borsuk-Ulam property. 

\medskip The study of BUP via sectional category is still non-existent and, in fact, this work takes a first step in this direction. Several examples are presented to illustrate the result arising in this field. We demonstrate that the triple $\left((X,\tau);Y\right)$ satisfies the BUP if $\mathrm{secat}\hspace{.1mm}(q)>\mathrm{secat}\hspace{.1mm}(q^Y)$ (Theorem~\ref{theorem-1}). As a result, we give an alternative proof of the fact that the triple $\left((S^m,A);\mathbb{R}^n\right)$ satisfies the BUP for any $1\leq n\leq m$ (Corollary~\ref{esfere-bup}). Moreover, we show that if $\mathrm{secat}(q:X\to X/\tau)>\text{Emb}(Y)$ then the triple $((X,\tau);Y)$ satisfies the BUP, where $\text{Emb}(M)$ is the smallest dimension of Euclidean spaces in which $M$ can be embedded (Proposition~\ref{embed}). For any planar graph $\Gamma$ such that $F(\Gamma,2)$ is path-connected, we show that the triple $\left((S^m,A);\Gamma\right)$ satisfies the BUP for any $m\geq 2$ (Example~\ref{planar-graph}). In addition, we study two natural generalizations of the Borsuk-Ulam theorem as follows.

The first natural generalization of the Borsuk-Ulam theorem consists in replacing $\mathbb{R}^n$ by a Hausdorff space $Y$, and then to ask which triples $\left((S^m,A);Y\right)$ satisfy the BUP. We obtain that if $\mathrm{secat}\left(q^Y:F(Y,2)\to D(Y,2) \right)<m+1$ then the triple $\left((S^m,A);Y\right)$ satisfies the BUP (Example~\ref{bup-sphere-y}). For $Y$ a path-connected topological manifold (without boundary), if $\dim(Y)\leq\dfrac{m}{2}$ we show that the triple $\left((S^{m+1},A);Y\right)$ satisfies the BUP (Proposition~\ref{bup-dim}). In particular, in Example~\ref{surface}, we show that the triple $\left((S^m,A);\Sigma\right)$ satisfies the BUP for any $m\geq 5$ and any connected surface $\Sigma$.

 The second natural generalization of the Borsuk-Ulam theorem consists in replacing $S^m$ by a connected, $m$-dimensional CW complex $M^m$ and $A$ by a free cellular involution $\tau$ defined on $M^m$, and then to ask which triples $\left((M^m,\tau);\mathbb{R}^n\right)$ satisfy the BUP. From \cite[Lemma 2.4]{goncalves2010} or Lemma~\ref{n-n+1}, if $n>m$ the BUP does not hold for $\left((M^m,\tau);\mathbb{R}^n\right)$. A major problem is to find the greatest $n\leq m$ such that the BUP holds for a specific $(M^m,\tau)$. Such greatest integer $n$ is known as the \textit{index} of $\tau$ on $M^m$ (Definition~\ref{defn-index}). We present a new lower bound for the index in terms of sectional category. Indeed, we demonstrate that the index of $\tau$ on $X$ is at least $\mathrm{secat}(q:X\to X/\tau)-1$ (Theorem~\ref{lower-bound-index}). This lower bound can be achieved. Corollary~\ref{m+1-m} shows that the index of $\tau$ on $M^{m}$ is $m$ when $\mathrm{secat}(q)=m+1$. Moreover, Proposition~\ref{m-1} shows that the index of $\tau$ on $M^m$ is $m-1$ when $\mathrm{secat}(q:M^m\to M^m/\tau)=m$. We show that the index of $(X,\tau)$ always coincides with the sectional category of the quotient map $q:X\to X/\tau$ minus 1 for any paracompact space $X$ (Theorem~\ref{thm:sec-index-one}). As an application of the  results, we present some new results relating  Borsuk-Ulam theory and sectional category. 

\medskip The paper is organized as follows: In Section \ref{bu}, we recall the notion of Borsuk-Ulam property. A key result in this paper is a topological characterisation for the BUP (it is given in Proposition~\ref{top-bup}) together with Remark~\ref{bup-pullback}. In Section \ref{sn}, we begin by recalling the notions of sectional category, L-S category, category of maps and basic results about these numerical invariants. We show that the equality $\mathrm{secat}(p\times 1_Z)=\mathrm{secat}(p)$ holds for any fibration $p:E\to B$ and any topological space $Z$ (Proposition~\ref{secat-produc}). In \cite{roth2008}, the author shows that $\mathrm{secat}(q^{\mathbb{R}^n})=n$ for any $n\geq 1$. We present an alternative proof of this fact (see Lemma~\ref{conf-eu}). In this section we study the BUP property for the triple $\left((X,\tau);Y\right)$ via sectional category. In particular, we demonstrate that the triple $\left((X,\tau);Y\right)$ satisfies the BUP if $\mathrm{secat}\hspace{.1mm}(q)>\mathrm{secat}\hspace{.1mm}(q^Y)$ (Theorem~\ref{theorem-1}). We give several examples which extend know results about the BUP. In particular, we recover the famous Borsuk-Ulam theorem (Corollary~\ref{esfere-bup}). In addition, in Lemma~\ref{n-n+1}, for connected topological manifolds $X$ and $Y$ with dimension $n$ and $n+1$, respectively (with $n\geq 1$); and a fixed-point free involution $\tau:X\to X$, we show that the triple $((X,\tau);Y)$ does not satisfy the BUP. As another application of our result, we present a new lower bound for the index in terms of sectional category (Theorem~\ref{lower-bound-index}). Furthermore, Theorem~\ref{thm:sec-index-one} shows that this lower bound can be achieved. In addition, we determine  the sectional category of a double covering $S_1 \to S_2$ between any two closed surfaces (Corollary~\ref{cor}).

\medskip The authors of this paper deeply thank the referees for very valuable comments and timely corrections on previous versions of the work.
\section{Borsuk-Ulam theory revisited}\label{bu}

\noindent Let $\left((X,\tau);Y\right)$ be a triple where $X$ is a Hausdorff space, $\tau:X\to X$ is a fixed-point free involution and $Y$ is a Hausdorff space. We say that $\left((X,\tau);Y\right)$ \textit{satisfies the Borsuk-Ulam property} (which we shall routinely abbreviate to BUP) if  for every continuous map $f:X\to Y$ there exists a point $x\in X$ such that $f(\tau(x))=f(x)$. 

Let $S^m$ be the $m$-dimensional sphere, $A:S^m\to S^m$ the antipodal involution 
 and $\mathbb{R}^n$ the $n$-dimensional Euclidean space. The famous Borsuk-Ulam theorem states that for every continuous map $f:S^m\to \mathbb{R}^m$ there exists a point $x\in S^m$ such that $f(x)=f(-x)$ \cite{borsuk1933}.
 
 \begin{remark}\label{subset}
Note that if $\left((X,\tau);Y\right)$ satisfies the BUP then the triple $\left((X,\tau);Z\right)$ also satisfies the BUP for any nonempty subspace $Z\subset Y$.
\end{remark}

A natural generalization of the Borsuk-Ulam theorem consists in replacing $S^m$ by a connected, $m$-dimensional CW complex $M^m$ and $A$ by a free cellular involution $\tau$ defined on $M^m$, and then to ask which triples $\left((M^m,\tau);\mathbb{R}^n\right)$ satisfy the BUP. From \cite[Lemma 2.4]{goncalves2010}, if $n>m$ the BUP does not hold for $\left((M^m,\tau);\mathbb{R}^n\right)$. A major problem is to find the greatest $n\leq m$ such that the BUP holds for a specific $(M^m,\tau)$. 

\begin{definition}\label{defn-index} Let  $(X, \tau)$ be a $\mathbb{Z}_2$-space. Following \cite[Definition 5.3.1]{mat}  
the $\mathbb{Z}_2$-index of $(X,\tau)$ is defined by:
$$\text{ind}_{\mathbb{Z}_2}(X):= \min\{n \in  \{0, 1, 2, \ldots \}: X\stackrel{\mathbb{Z}_2} {\to}  
 S^n\}.$$
\end{definition}

From \cite[Propostion 2.2]{goncalves2010} follows that the greatest $n$ such that 
 $\left((M^m,\tau);\mathbb{R}^n\right)$ satisfy the BUP coincides with $\text{ind}_{\mathbb{Z}_2}(M^m)$, 
 where the $\mathbb{Z}_2$ action is given by $\tau$.

We will fix some notation that will be used throughout the paper. The \textit{ordered configuration space} of $2$ distinct points on $Y$ (see \cite{fadell1962configuration}) is the topological space \[F(Y,2)=\{(y_1,y_2)\in Y\times Y:~y_1\neq y_2\}\] topologised as a subspace of the Cartesian power $Y\times Y$. Consider the double cover $q^Y:F(Y,2)\to D(Y,2)$ from the ordered configuration space $F(Y,2)$ to its unordered quotient $D(Y,2)=F(Y,2)/\Sigma_2$ given by the obvious action $\tau_2$ of the symmetric group $\Sigma_2$ on $2$ letters. Note that the existence of a free action of $\mathbb{Z}_2$ on $X$ is equivalent to that of a fixed-point free involution $\tau:X\to X$. In this case, with $X$ Hausdorff, the quotient map $q:X\to X/\tau$ is a $2$-sheeted covering map. Recall that a covering map is a locally trivial bundle whose fiber is a discrete space \cite[Example 4.5.3, pg. 126]{aguilar2002} and thus the quotient map $q:X\to X/\tau$ is a principal fibering in the sense of Schwarz \cite[Pg. 59]{schwarz1966}.  

It is easy to check the following topological criterion for the BUP (c.f. \cite[Lemma 5]{daciberg2019}). 

\begin{proposition}\label{top-bup} 
The triple $\left((X,\tau);Y\right)$ does not satisfy the BUP if and only if there exists a $\mathbb{Z}_2$-equivariant continuous map $\varphi:X\to F(Y,2)$ such that the following diagram commutes
\begin{eqnarray*}
\xymatrix@C=2cm{ 
       X  \ar[r]^{\varphi}\ar[d]_{q} &  F(Y,2)\ar[d]^{q^Y} &\\
      X/\tau \ar[r]_{\overline{\varphi}} & D(Y,2) &
       }
\end{eqnarray*} where $q:X\to X/\tau$ and $q^Y:F(Y,2)\to D(Y,2)$ are the $2$-sheeted covering maps, and $\overline{\varphi}$ is induced by $\varphi$ in the quotient spaces. 
\end{proposition}
\begin{proof}
Suppose first that $\left((X,\tau);Y\right)$ does not satisfy the BUP. Then there exists a map $f:X\to Y$ such that $f(x)\neq f(\tau(x))$ for all $x\in X$. Define the map $\varphi:X\to F(Y,2)$ by $\varphi(x)=(f(x),f(\tau (x)))$. Note that $\varphi$ is $\mathbb{Z}_2$-equivariant, and so induces a map $\overline{\varphi}$ of the corresponding quotient spaces. Moreover, the equality $\overline{\varphi}\circ q=q^Y\circ \varphi$ holds.  

We now prove the converse. Suppose that there exists a such $\mathbb{Z}_2$-equivariant map $\varphi:X\to F(Y,2)$. Let $\varphi=(\varphi_1,\varphi_2)$, and note that $\varphi(\tau(x))=(\varphi_2(x),\varphi_1(x))$ for all $x\in X$, and so $\varphi_1(\tau(x))=\varphi_2(x)$ for all $x\in X$. Then $\varphi_1:X\to Y$ is a map with $\varphi_1(x)\neq  \varphi_2(x)=\varphi_1(\tau(x))$ for all $x\in X$, and we have that $\left((X,\tau);Y\right)$ does not satisfy the BUP. 
\end{proof}

From \cite[Pg. 61]{schwarz1966} we have the following remark. 

\begin{remark}\label{bup-pullback} Let $X, Y$ be Hausdorff spaces. Suppose that $X$ admits a fixed-point free involution $\tau$. Note that, any commutative diagram in the form  \begin{eqnarray*}
\xymatrix{ 
       X  \ar[r]^{\varphi}\ar[d]_{q} &  F(Y,2)\ar[d]^{q^Y} &\\
      X/\tau \ar[r]_{\overline{\varphi} } & D(Y,2) &
       }
\end{eqnarray*} where $\varphi:X\to F(Y,2)$ is a continuous $\mathbb{Z}_2$-equivariant map and $\overline{\varphi}$ is induced by $\varphi$ in the quotient spaces, is a pullback since $\varphi$ restricts to a homeomorphism on each fiber (in this case both $q:X\to X/\tau$ and $q^Y:F(Y,2)\to D(Y,2)$ are $2$-sheeted covering maps).
\end{remark}

\section{Sectional category}\label{sn}
In this section we begin by recalling the notion of sectional category together with basic results about this numerical invariant. We shall follow the terminology in \cite{zapata2020}. If $f$ is homotopic to $g$ we shall denote by $f\simeq g$. The map $1_Z:Z\to Z$ denotes the identity map. Fibrations are taken in the Hurewicz sense. 

\medskip Let $p:E\to B$ be a fibration.  A \textit{cross-section} or \textit{section} of $p$ is a right inverse of $p$, i.e., a map $s:B\to E$, such that $p\circ s = 1_B$ . Moreover, given a subspace $A\subset B$, a \textit{local section} of $p$ over $A$ is a section of the restriction map $p_|:p^{-1}(A)\to A$, i.e., a map $s:A\to E$, such that $p\circ s$ is the inclusion $A\hookrightarrow B$.

\medskip We recall the following definition, see \cite{schwarz1958genus} or \cite{schwarz1966}.
\begin{definition}
   The \textit{sectional category} of $p$, called originally by Schwarz genus of $p$, and denoted by $\mathrm{secat}\hspace{.1mm}(p),$ is the minimal cardinality of open covers of $B$, such that each element of the cover admits a continuous local section to $p$. We set $\mathrm{secat}\hspace{.1mm}(p)=\infty$ if no such finite cover exists.
\end{definition}

For a commutative ring $R$ and a proper ideal $S\subset R$, the \textit{nilpotency index} of $S$ is given by \[\text{nil}(S)=\min\{k:~\text{ any product of $k$ elements of $S$ is trivial}\}.\] Note that, $\text{nil}(S)$ coincides with $n+1$, where $n$ is the maximum number of factors in a nonzero product of elements from $S$.

The following statement gives a lower bound in terms of any multiplicative cohomology (see \cite[Proposiç\~{a}o 4.3.17-(3), pg. 138]{zapata2022}).

\begin{lemma}\label{prop-sectional-category}
Let $h^\ast$ be a multiplicative cohomology theory and $p:E\to B$ be a fibration, then \[\mathrm{secat}\hspace{.1mm}(p)\geq \mathrm{nil}(\mathrm{Ker}(p^\ast)),\] where $p^\ast:h^\ast(B)\to h^\ast(E)$ is the induced homomorphism in cohomology.
\end{lemma}

\begin{remark}
Lemma~\ref{prop-sectional-category} implies that if there exist cohomology classes $\alpha_1,\ldots,\alpha_k\in h^\ast(B)$ with $p^\ast(\alpha_1)=\cdots=p^\ast(\alpha_k)=0$ and $\alpha_1\cup\cdots\cup \alpha_k\neq 0$, then $\mathrm{secat}\hspace{.1mm}(p)\geq k+1$. In this paper we will use Lemma~\ref{prop-sectional-category} for $h^\ast$ as being the singular cohomology with any coefficient ring (as was presented by James in \cite[pg. 342]{james1978}).
\end{remark} 

\medskip\noindent Now, note that, if the following diagram

\begin{eqnarray*}
\xymatrix{ E^\prime \ar[rr]^{\,\,} \ar[dr]_{p^\prime} & & E \ar[dl]^{p}  \\
        &  B & }
\end{eqnarray*}
commutes up homotopy, then $\mathrm{secat}\hspace{.1mm}(p^\prime)\geq \mathrm{secat}\hspace{.1mm}(p)$ (see \cite[Proposition 6, pg. 70]{schwarz1966}). Also, from \cite[Proposition 7, pg. 71]{schwarz1966}, for any fibration $p:E\to B$ and any  continuous map $f:B^\prime\to B$, note that any local section $s:U\to E$ of $p:E\to B$ induces a local section of the canonical pullback $f^\ast p:B^\prime\times_B E\to B^\prime$, called \textit{the local pullback section} $f^\ast(s):f^{-1}(U)\to B^\prime\times_B E$, simply by defining \[f^\ast(s)(b^\prime)=\left(b^\prime,\left(s\circ f\right)(b^\prime)\right).\]

\begin{eqnarray*}
\xymatrix{ &B^\prime\times_B E \ar[rr]^{ } \ar[d]_{f^\ast p} & & E \ar[d]^{p} & \\
       &B^\prime  \ar[rr]_{f} & &  B & \\
       f^{-1}(U)\ar@{^{(}->}[ru]_{}\ar[rr]_{f}\ar@{-->}@/^10pt/[ruu]^{f^\ast(s)} & &U\ar@{^{(}->}[ru]_{}\ar@{-->}@/^10pt/[ruu]^{s} & & }
\end{eqnarray*} Thus, \begin{align}\label{ineq-canonical}
    \text{secat}(f^\ast p)&\leq \text{secat}(p).
\end{align} 

Before stating Lemma~\ref{pullback}, let us present a key concept which we call  {\it  quasi pullback}. 

\begin{definition}
 By a \textit{quasi pullback} we mean a strictly commutative diagram
\begin{eqnarray*}
\xymatrix{ \rule{3mm}{0mm}& X^\prime \ar[r]^{\varphi'} \ar[d]_{f^\prime} & X \ar[d]^{f} & \\ &
       Y^\prime  \ar[r]_{\,\,\varphi} &  Y &}
\end{eqnarray*} 
such that, for any strictly commutative diagram as the one on the left hand-side of~(\ref{diagramadoble}), there exists a (not necessarily unique) continuous map $h:Z\to X^\prime$ that renders a strictly commutative diagram as the one on the right hand-side of~(\ref{diagramadoble}). 
\begin{eqnarray}\label{diagramadoble}
\xymatrix{
Z \ar@/_10pt/[dr]_{\alpha} \ar@/^30pt/[rr]^{\beta} & & X \ar[d]^{f}  & & &
Z\rule{-1mm}{0mm} \ar@/_10pt/[dr]_{\alpha} \ar@/^30pt/[rr]^{\beta}\ar[r]^{h} & 
X^\prime \ar[r]^{\varphi'} \ar[d]_{f^\prime} & X \\
& Y^\prime  \ar[r]_{\,\,\varphi} &  Y & & & & Y^\prime &  \rule{3mm}{0mm}}
\end{eqnarray}   
\end{definition}

Note that such a condition amounts to saying that $X'$ contains the canonical pullback $Y'\times_Y X$ determined by $f$ and $\varphi$ as a retract in a way that is compatible with the mappings into $X$ and $Y'$.

Then we have the following statement.

\begin{lemma}\label{pullback}
Let $p:E\to B$ be a fibration. If the following square
\begin{eqnarray*}
\xymatrix{ E^\prime \ar[r]^{\,\,} \ar[d]_{p^\prime} & E \ar[d]^{p} & \\
       B^\prime  \ar[r]_{\,\, f} &  B &}
\end{eqnarray*}
is a quasi pullback, then $\mathrm{secat}\hspace{.1mm}(p^\prime)\leq \mathrm{secat}\hspace{.1mm}(p)$. 
\end{lemma}
\begin{proof}
Since $p'$ is a quasi pullback, we have the following commutative triangle
\begin{eqnarray*}
\xymatrix{ B^\prime\times_B E \ar[rr]^{\,\,} \ar[dr]_{f^\ast p} & & E^\prime \ar[dl]^{p^\prime}  \\
        &  B^\prime & }
\end{eqnarray*} and thus $\text{secat}(f^\ast p)\geq\text{secat}(p')$. Similarly, since $f^\ast p$ is the canonical pullback, we have the following commutative triangle
\begin{eqnarray*}
\xymatrix{ E^\prime \ar[rr]^{\,\,} \ar[dr]_{p^\prime} & &  B^\prime\times_B E\ar[dl]^{f^\ast p}  \\
        &  B^\prime & }
\end{eqnarray*} and thus $\text{secat}(p')\geq\text{secat}(f^\ast p)$. Hence, the equality $\text{secat}(p')=\text{secat}(f^\ast p)$ holds. By the inequality~(\ref{ineq-canonical}), we obtain $\mathrm{secat}\hspace{.1mm}(p^\prime)\leq \mathrm{secat}\hspace{.1mm}(p)$.
\end{proof}

In Example~\ref{example-final} we will use the equality $\mathrm{secat}(p\times 1_Z)=\mathrm{secat}(p)$ which holds for any fibration $p:E\to B$ and any topological space $Z$. For that reason, we present a proof to it.

\begin{proposition}\label{secat-produc}
Let $p:E\to B$ be a fibration and $Z$ be a topological space, then \[\mathrm{secat}(p\times 1_Z)=\mathrm{secat}(p).\]
\end{proposition}
\begin{proof}
Note that, if $s:U\to E$ is a section to $p$, then the product $s\times 1_Z:U\times Z\to E\times Z$ is a section to $p\times 1_Z$, and thus, $\mathrm{secat}(p\times 1_Z)\leq\mathrm{secat}(p)$. The other inequality follows from the fact that the square 
\begin{eqnarray*}
\xymatrix@C=2cm{ E \ar[r]^{j_{z_0}} \ar[d]_{p} & E\times Z \ar[d]^{p\times 1_Z} & \\
       B \ar[r]_{j_{z_0}} &  B\times Z &}
\end{eqnarray*} where $j_{z_0}(-)=(-,z_0)$ is the natural inclusion ($z_0\in Z$), 
is a quasi-pullback together with Lemma~\ref{pullback}. 
\end{proof}

Next, we recall the notion of LS category which, in our setting, is one bigger than the one given in \cite[Definition 1.1, pg.1]{cornea2003lusternik}. 

\begin{definition}
The \textit{Lusternik-Schnirelmann category} (L-S category) or category of a topological space $X$, denoted by cat$(X)$, is the least integer $m$ such that $X$ can be covered by $m$ open sets, all of which are contractible within $X$. We set $\text{cat}(X)=\infty$ if no such $m$ exists.
\end{definition}

We have $\text{cat}(X)=1$ iff $X$ is contractible. The L-S category is a homotopy invariant, i.e., if $X$ is homotopy equivalent to $Y$ (which we shall denote by $X\simeq Y$), then $\text{cat}(X)=\text{cat}(Y)$. Furthermore, the invariant satisfies the following properties.

\begin{lemma}\label{cat-stimates}
\noindent
\begin{enumerate}
    \item \cite[Proposition 5.1, pg. 336]{james1978} If $X$ is a $(q-1)$-connected CW complex ($q\geq 1$), then \[\mathrm{cat}(X)\leq \dfrac{\mathrm{hdim}(X)}{q}+1,\]  where $\mathrm{hdim}(X)$ denotes the homotopical dimension of $X$, i.e., the minimal dimension of CW complexes having the homotopy type of $X$. 
    \item \cite[Proposiç\~{a}o 4.1.34, pg. 108]{zapata2022} We have $$\mathrm{cat}(X)\geq \mathrm{nil}\left(\widetilde{h}^\ast(X)\right),$$ where $\widetilde{h}^\ast(X)$ is any multiplicative reduced cohomology theory.
\end{enumerate}
\end{lemma}

\medskip We recall the following statements.

 \begin{lemma}\label{prop-secat-map}
 Let $p:E\to B$ be a fibration.
 \begin{enumerate}
     \item \cite[Theorem 18, pg. 108]{schwarz1966} We have $\mathrm{secat}\hspace{.1mm}(p)\leq \mathrm{cat}(B)$. 
     \item  \cite[Proposição 4.3.17, pg. 138]{zapata2022} If $p$ is nulhomotopic, then $\mathrm{secat}\hspace{.1mm}(p)=\mathrm{cat}(B).$
 \end{enumerate}
\end{lemma}

\begin{remark}\label{trivial-crosssection} Given a Hausdorff space $X$ admitting a fixed-point free involution $\tau$, we have that:
\begin{itemize}
    \item[i)] The quotient map $q\colon X\to X/\tau$ is a fibration \cite[Theorem 3.2.2, pg. 66]{tom2008}. Therefore, it is possible to speak of its sectional category.  
    \item[ii)] The quotient map $q:X\to X/\tau$ is the trivial bundle if and only if it admits a continuous cross-section \cite[Pg. 36]{steenrod1951}. In particular, in the case when $X$ is path-connected, the quotient map $q:X\to X/\tau$ is not the trivial bundle and does not admit a continuous cross-section.
\end{itemize}
     
\end{remark}

\begin{example}\label{projective}
Recall that the $\mathbb{Z}_2$-cohomology of $\mathbb{R}P^m$ ($m\geq 1$) is given by $H^\ast(\mathbb{R}P^m;\mathbb{Z}_2)=\dfrac{\mathbb{Z}_2[\alpha]}{\langle \alpha^{m+1}\rangle}$ with $0\neq\alpha\in H^1(\mathbb{R}P^m;\mathbb{Z}_2)$. Then, for  dimensional reasons, the induced homomorphism $q^\ast_{\mathbb{Z}_2}:H^\ast\left(\mathbb{R}P^m;\mathbb{Z}_2\right)\to H^\ast\left(S^m;\mathbb{Z}_2\right)$ is trivial  and thus $\text{Ker}(q^\ast_{\mathbb{Z}_2})=\widetilde{H}^\ast(\mathbb{R}P^m;\mathbb{Z}_2)$ for any $m\geq 2$. Then $\mathrm{Nil}\left(\text{Ker}(q^\ast_{\mathbb{Z}_2})\right)\geq m+1$ for any $m\geq 2$. In addition $\text{cat}(\mathbb{R}P^m)=m+1$ (see \cite[Example 1.8, pg.4]{cornea2003lusternik}). Thus, for any $m\geq 2$, we have \[\mathrm{secat}(q:S^m\to \mathbb{R}P^m)=\mathrm{cat}(\mathbb{R}P^m)=\mathrm{Nil}\left(\text{Ker}(q^\ast_{\mathbb{Z}_2})\right)=m+1,\] where $q^\ast_{\mathbb{Z}_2}:H^\ast\left(\mathbb{R}P^m;\mathbb{Z}_2\right)\to H^\ast\left(S^m;\mathbb{Z}_2\right)$ is the induced homomorphism in $\mathbb{Z}_2$-cohomology. Furthermore, $\mathrm{secat}(q:S^1\to \mathbb{R}P^1)=\mathrm{cat}(\mathbb{R}P^1)=2$. Indeed, for $m\geq 2$, by Lemma~\ref{prop-sectional-category} together with Lemma~\ref{prop-secat-map} we have \begin{align*}
    m+1&\leq \mathrm{Nil}\left(\text{Ker}(q^\ast_{\mathbb{Z}_2})\right)\\
    &\leq \mathrm{secat}(q:S^m\to \mathbb{R}P^m)\\
    &\leq \mathrm{cat}(\mathbb{R}P^m)\\
    &= m+1.
\end{align*} On the other hand, $\mathrm{secat}(q:S^1\to \mathbb{R}P^1)\geq 2$. Moreover, again by Lemma~\ref{prop-secat-map}, $\mathrm{secat}(q:S^1\to \mathbb{R}P^1)\leq \mathrm{cat}(\mathbb{R}P^1)=2$.
\end{example}

In \cite{roth2008}, the author shows that $\mathrm{secat}(q^{\mathbb{R}^n})=n$ for any $n\geq 1$. However, we present an alternative proof of this fact.

\begin{lemma}\label{conf-eu}
We have that $\mathrm{secat}\left(q^{\mathbb{R}^n}:F(\mathbb{R}^n,2)\to D(\mathbb{R}^n,2)\right)=n$ for any $n\geq 1$.
\end{lemma}
\begin{proof}
The case $n=1$ is easy since the configuration space $D(\mathbb{R},2)$ is contractible. For $n\geq 2$, consider the maps $\varphi\colon S^{n-1}\to F(\mathbb{R}^n,2)$ and $\psi\colon F(\mathbb{R}^n,2)\to S^{n-1}$ given by $\varphi(x)=(x,-x)$ and $\psi(x,y)=\dfrac{x-y}{\parallel x-y\parallel}$. Then, we have that the following diagrams \begin{eqnarray*}
\xymatrix@C=1cm{ 
       S^{n-1}  \ar[r]^{\varphi}\ar[d]_{q} &  F(\mathbb{R}^n,2)\ar[d]^{q^{\mathbb{R}^n}} &\\
      \mathbb{R}P^{n-1} \ar[r]_{\overline{\varphi}} & D(\mathbb{R}^n,2) &
       } & \xymatrix@C=1cm{ 
       F(\mathbb{R}^n,2) \ar[r]^{\quad\psi} \ar[d]_{q^{\mathbb{R}^n}} &  S^{n-1}\ar[d]^{q} &\\
      D(\mathbb{R}^n,2) \ar[r]_{\quad\overline{\psi}} &  \mathbb{R}P^{n-1} &
       }
\end{eqnarray*} are pullbacks. Then, by Lemma~\ref{pullback}, we conclude that $\mathrm{secat}(q^{\mathbb{R}^n})=\mathrm{secat}(q)$ and therefore, $\mathrm{secat}(q^{\mathbb{R}^n})=n$ (by Example~\ref{projective}).
\end{proof}

In the same way, we will check that $\mathrm{secat}\left(q^{S^n}\right)=n+1$ for any $n\geq 1$. 

\begin{lemma}\label{conf-sphere}
We have that $\mathrm{secat}\left(q^{S^n}:F(S^n, 2)\to D(S^n,2)\right)=n+1$ for any $n\geq 1$.
\end{lemma}
\begin{proof}
Consider the following pullbacks \begin{eqnarray*}
\xymatrix@C=1cm{ 
       S^{n}  \ar[r]^{\varphi}\ar[d]_{q} &  F(S^n,2)\ar[d]^{q^{S^n}} &\\
      \mathbb{R}P^{n} \ar[r]_{\overline{\varphi}} & B(S^n,2) &
       } & \xymatrix@C=1cm{ 
       F(S^n,2) \ar[r]^{\quad\psi} \ar[d]_{q^{S^n}} &  S^{n}\ar[d]^{q} &\\
      B(S^n,2) \ar[r]_{\quad\overline{\psi}} &  \mathbb{R}P^{n} &
       }
\end{eqnarray*} where $\varphi(x)=(x,-x)$ and $\psi(x,y)=\dfrac{x-y}{\parallel x-y\parallel}$. Then, by Lemma~\ref{pullback}, we conclude that $\mathrm{secat}(q^{S^n})=\mathrm{secat}(q)$ and therefore, $\mathrm{secat}(q^{S^n})=n+1$ (by Example~\ref{projective}).
\end{proof}

Our main result is as follows. We show that the BUP is valid if certain conditions in terms of
sectional category are satisfied.
\begin{theorem}[Principal theorem]\label{theorem-1} Suppose that $X$ and $Y$ are Hausdorff spaces with a fixed-point free involution $\tau:X\to X$. If the triple $\left((X,\tau);Y\right)$ does not satisfy the BUP then \[\mathrm{secat}(q)\leq \mathrm{secat}(q^Y).\] Equivalently, if $\mathrm{secat}(q)> \mathrm{secat}(q^Y)$ then the triple $((X,\tau);Y)$ satisfies the BUP.
\end{theorem}
\begin{proof}
It follows from Proposition~\ref{top-bup} and Remark~\ref{bup-pullback} together with Lemma~\ref{pullback}.
\end{proof}

Next, we will present direct applications of Theorem~\ref{theorem-1}.
\begin{example}\label{bup-sphere-y}
Let $m\geq 1$ and $Y$ be a Hausdorff space, and consider the covering maps: 
\begin{eqnarray*}
\xymatrix{ S^m  \ar[d]_{q} &  F(Y,2)\ar[d]^{q^{Y}} &\\
      \mathbb{R}P^m  & D(Y,2) &
       }
\end{eqnarray*} Recall that $\mathrm{secat}(q)=m+1$ (see Example~\ref{projective}). Therefore, by Theorem~\ref{theorem-1}, if $\mathrm{secat}(q^Y)<m+1$ then the triple $\left((S^m,A);Y\right)$ satisfies the BUP. 
\end{example}

Example~\ref{bup-sphere-y} implies the famous Borsuk-Ulam theorem \cite{borsuk1933}.

\begin{corollary}\label{esfere-bup}[Famous Borsuk-Ulam theorem]
We have that the triple $\left((S^m,A);\mathbb{R}^n\right)$ satisfies the BUP for any $1\leq n\leq m$.
\end{corollary}
\begin{proof}
    We have that $\mathrm{secat}(q^{\mathbb{R}^n})=n$ (see Lemma~\ref{conf-eu}), and thus, from Example~\ref{bup-sphere-y}, we conclude that the triple $\left((S^m,A);\mathbb{R}^n\right)$ satisfies the BUP for any $1\leq n\leq m$. 
\end{proof}

Moreover, we have the following statement.

\begin{proposition}\label{bup-dim} 
Let $Y$ be a path-connected topological manifold (without boundary) such that $\dim(Y)\leq\dfrac{m}{2}$. Then 
\begin{enumerate}
    \item The triple $\left((S^{m+1},A);Y\right)$ satisfies the BUP.
    \item If $\mathrm{hdim}(D(Y,2))\leq 2\dim(Y)-1$ then, $\left((S^{m},A);Y\right)$ satisfies the BUP.
\end{enumerate}
\end{proposition}
\begin{proof} The case, $\dim(Y)\leq 1$ is shown easily. We will suppose that $\dim(Y)\geq 2$.
\begin{enumerate}
    \item From Lemma~\ref{prop-secat-map} together with Lemma~\ref{cat-stimates}-item (1), we have \begin{eqnarray*}
\mathrm{secat}(q^Y) &\leq & \mathrm{cat}(D(Y,2))\\
&\leq& \mathrm{hdim}(D(Y,2))+1\\
&\leq& 2\dim(Y)+1\\
&\leq& m+1.
\end{eqnarray*} Thus $\mathrm{secat}(q^Y)<m+2=\mathrm{secat}(q\colon S^{m+1}\to\mathbb{R}P^{m+1})$ and by Example~\ref{bup-sphere-y}, then we conclude that the triple $\left((S^{m+1},A);Y\right)$ satisfies the BUP. 
\item From Lemma~\ref{prop-secat-map} together with Lemma~\ref{cat-stimates}-item (1), we have \begin{eqnarray*}
\mathrm{secat}(q^Y) &\leq & \mathrm{cat}(D(Y,2))\\
&\leq& \mathrm{hdim}(D(Y,2))+1\\
&\leq& 2\dim(Y)-1+1\\
&\leq& m.
\end{eqnarray*} Thus $\mathrm{secat}(q^Y)<m+1=\mathrm{secat}(q\colon S^{m}\to\mathbb{R}P^{m})$ and by Example~\ref{bup-sphere-y}, then we conclude that the triple $\left((S^{m},A);Y\right)$ satisfies the BUP. 
\end{enumerate}
\end{proof}

We do not know an example of a path-connected manifold $Y$ (without boundary) such that
$\mathrm{hdim}(D(Y,2))=2\dim(Y)$.

An immediate consequence of Proposition~\ref{bup-dim} is the following example.

\begin{example}\label{surface}
Let $\Sigma$ be a connected surface. Then the triple $\left((S^m,A);\Sigma\right)$ satisfies the BUP for any $m\geq 5$.
\end{example}

Example~\ref{surface} together with \cite[Theorem 2, pg. 1743]{daciberg2010} imply that the triple $\left((S^m,A);\mathbb{R}P^2\right)$ satisfies the BUP for any $m\geq 2$. In \cite[Theorem 2, pg. 1743]{daciberg2010}, the authors showed that the triple $\left((S^m,A);\mathbb{R}P^2\right)$ satisfies the BUP for $m\in\{2,3\}$. On the other hand, we can check that (or see \cite[Proposition 4, pg. 1743]{daciberg2010}) the triple $\left((S^m,A);\Sigma\right)$ satisfies the BUP for any $m\geq 2$ and any connected closed surface $\Sigma$ different from $\mathbb{R}P^2$ and $S^2$.

\medskip  Let $\text{Emb}(M)$ be the smallest dimension of Euclidean spaces when $M$ can be embedded. We set $\text{Emb}(M)=\infty$ if no such embedding exists.  We have the following statement. 

\begin{proposition}\label{embed}
Suppose that $X$ and $Y$ are Hausdorff spaces with $\mathrm{Emb}(Y)<\infty$. If $\mathrm{secat}(q:X\to X/\tau)>\mathrm{Emb}(Y)$ then the triple $((X,\tau);Y)$ satisfies the BUP.
\end{proposition}
\begin{proof}
Consider $Y\subset\mathbb{R}^k$, where $k=\mathrm{Emb}(Y)$. We have the following pullback 
\begin{eqnarray*}
\xymatrix{ F(Y,2)  \ar[d]_{q^Y}\ar[r]_{ } &  F(\mathbb{R}^k,2)\ar[d]^{q^{\mathbb{R}^k}} &\\
      D(Y,2)\ar[r]_{ }  & D(\mathbb{R}^k,2) &
       }
\end{eqnarray*} where the horizontal maps in the diagram are induced by
the inclusion $Y\hookrightarrow\mathbb{R}^k$. Then $\mathrm{secat}(q^Y)\leq\mathrm{secat}(q^{\mathbb{R}^k})=k=\mathrm{Emb}(Y)<\mathrm{secat}(q:X\to X/\tau)$. Thus, by Theorem~\ref{theorem-1}, we obtain that the triple $((X,\tau);Y)$ satisfies the BUP. 
\end{proof}

In particular, we have the following example. 

\begin{example}\label{sphere-orientable-surface}
Note that $\mathrm{Emb}(S^2)=3$. Thus, by Proposition~\ref{embed}, we obtain that the triple $\left((S^m,A);S^2\right)$ satisfies the BUP for any $m\geq 3$. 
\end{example}

Note that, Example~\ref{sphere-orientable-surface} cannot be improved, that is, we can check that (or see \cite[Corollary 1-(b), pg. 1743]{daciberg2010}) the triple $\left((S^2,A);S^2\right)$ does not  satisfy the BUP.

\medskip The following statement is a partial result of the BUP when $Y=\Gamma$ is a graph. 
\begin{example}\label{planar-graph}
Suppose that $\Gamma$ is a planar graph such that $F(\Gamma,2)$ is path-connected. We have that $\mathrm{Emb}(\Gamma)=2$. Thus, by Proposition~\ref{embed}, we obtain that the triple $\left((S^m,A);\Gamma\right)$ satisfies the BUP for any $m\geq 2$. Moreover, note that $\mathrm{secat}(q^{\Gamma})=2$. Indeed, the inequality $2\leq \mathrm{secat}(q^{\Gamma})$ follows from Remark~\ref{trivial-crosssection} item $(i)$ together with the hypothesis that $F(\Gamma,2)$ is path-connected. The inequality $ \mathrm{secat}(q^{\Gamma})\leq 2$ follows from the hypothesis that $\Gamma$ is planar (hence $\mathrm{secat}(q^{\Gamma})\leq \mathrm{secat}(q^{\mathbb{R}^2})=2$).  
\end{example}

\begin{remark}
Note that, for $m\geq 1$, by the famous Borsuk-Ulam theorem (Example~\ref{esfere-bup}) together with Remark~\ref{subset}, we obtain that the triple $\left((S^m,A);Z\right)$ satisfies the BUP for any subspace $Z\subset\mathbb{R}^m$.
\end{remark}

When $S^\infty$ is the infinite dimensional sphere we have:

\begin{example}
For cohomological reasons, note that $\mathrm{secat}(S^\infty\to\mathbb{R}P^\infty)=\infty$, then the triple $\left((S^\infty,A);Y\right)$ satisfies the BUP for any finite dimensional topological manifold $Y$. Indeed, we have \begin{eqnarray*}
\mathrm{secat}(q^Y) &\leq& \mathrm{cat}(D(Y,2)) \\ &\leq&\dim\left(D(Y,2)\right)+1\\ &<& \infty.
\end{eqnarray*}
\end{example}

Next, we will study the BUP when $X=F(Z,2)$ is the ordered configuration space. Recall that, we have the free involution $\tau_2:F(Z,2)\to F(Z,2),~\tau_2(x,y)=(y,x)$ and the equality $\mathrm{secat}(q^{\mathbb{R}^n})=n$ (see Lemma~\ref{conf-eu}).

\begin{proposition}\label{domain-conf-eucli} 
Let $Y$ be a path-connected topological manifold (without boundary). If $\dim(Y)\leq\dfrac{n-1}{2}$ then the triple $\left((F(\mathbb{R}^{n+1},2),\tau_2);Y\right)$ satisfies the BUP.
\end{proposition}
\begin{proof}
Let $h:F(\mathbb{R}^{n+1},2)\to Y$ be any continuous map. Consider $j:S^{n}\to F(\mathbb{R}^{n+1},2)$ given by $j(x)=(x,-x)$ and we have the composition $h\circ j:S^{n}\to Y$. Then, by Proposition~\ref{bup-dim} item (1), there exists $x\in S^{n}$ such that $h\circ j(-x)=h\circ j(x)$, and thus $h(-x,x)=h(x,-x)$. Then, there is a point $z=(x,-x)\in F(\mathbb{R}^{n+1},2)$ such that $h(\tau_2(z))=h(z)$. Therefore, the triple $\left((F(\mathbb{R}^{n+1},2),\tau_2);Y\right)$ satisfies the BUP.
\end{proof}

In particular, when $Y$ is a surface, we have the following statement.

\begin{example}
Let $\Sigma$ be a path-connected surface, then the triple $\left((F(\mathbb{R}^n,2),\tau_2);\Sigma\right)$ satisfies the BUP for any $n\geq 6$ (by Proposition~\ref{domain-conf-eucli}). Furthermore, by Proposition~\ref{embed}, we obtain that $\left((F(\mathbb{R}^n,2),\tau_2);\Sigma\right)$ satisfies the BUP for any $n\geq 4$ when $\Sigma$ is a connected orientable surface.
\end{example}

Lemma~\ref{conf-sphere} says that the equality $\mathrm{secat}(q^{S^n})=n+1$ holds for any $n\geq 1$. Thus, we have the BUP when $X=F(S^n,2)$.
\begin{proposition}\label{domain-conf-sphere}
Let $Y$ be a path-connected topological manifold (without boundary). If $\dim(Y)\leq\dfrac{n}{2}$ then the triple $\left((F(S^{n+1},2),\tau_2);Y\right)$ satisfies the BUP.
\end{proposition}
\begin{proof}
Let $h:F(S^{n+1},2)\to Y$ be any continuous map. Consider $j:S^{n+1}\to F(S^{n+1},2)$ given by $j(x)=(x,-x)$ and we have the composition $h\circ j:S^{n+1}\to Y$. Then, by Proposition~\ref{bup-dim}, there exists $x\in S^{n+1}$ such that $h\circ j(-x)=h\circ j(x)$, and thus $h(-x,x)=h(x,-x)$. Then, there is a point $z=(x,-x)\in F(S^{n+1},2)$ such that $h(\tau_2(z))=h(z)$. Therefore, the triple $\left((F(S^{n+1},2),\tau_2);Y\right)$ satisfies the BUP.
\end{proof}

One more time, when $Y$ is a surface, we have the following statement.
\begin{example}
Let $\Sigma$ be a path-connected surface. By Proposition~\ref{domain-conf-sphere}, then the triple $\left((F(S^n,2),\tau_2);\Sigma\right)$ satisfies the BUP for any $n\geq 5$. On the other hand, by Proposition~\ref{embed}, we obtain $\left((F(S^n,2),\tau_2);\Sigma\right)$ satisfies the BUP for any $n\geq 3$ when $\Sigma$ is a connected orientable surface.
\end{example}

The following statement generalises \cite[Lemma 2.4]{goncalves2010}. Note that for any principal $\mathbb{Z}_2$-bundle $q:M^m\to M^m/\tau$, we will write $f_q:M^m/\tau\to\mathbb{R}P^\infty=B\mathbb{Z}_2$ for the classifying map of the bundle $q$. It is unique up to homotopy.

\begin{lemma}\label{n-n+1}
Let $X$ be a path-connected CW complex and $Y$ be a path-connected topological manifold with dimension $n$ and $n+1$, respectively (with $n\geq 1$); and $\tau:X\to X$ be a fixed-point free cellular involution. Then $((X,\tau);Y)$ does not satisfy the BUP. 
\end{lemma}
\begin{proof}
Since $\dim(X)=n$ and thus $\dim(X/\tau)=n$, there is a map $\overline{\rho}\colon X/\tau\to \mathbb{R}P^n$ such that the triangle 
\begin{eqnarray*}
\xymatrix{ X\ar[d]_{q}& S^n\ar[d]_{q'} \\
X/\tau\ar[r]^{\overline{\rho}}\ar[rd]_{f_q}&\mathbb{R}P^n\ar@{^{(}->}[d]_-{ }   \\ 
        & \mathbb{R}P^{\infty} & & \\
       }
\end{eqnarray*} commutes up homotopy. Then, there is a $\mathbb{Z}_2$-equivariant continuous map $\rho\colon X\to S^n$ such that $q'\circ\rho=\overline{\rho}\circ q$. Let $\mathbb{R}^{n+1}\subset Y$ be an embedding. Thus, we have the following commutative diagram:
\begin{eqnarray*}
\xymatrix@C=2cm{ X\ar[d]_{q}\ar[r]^{\rho}& S^n\ar[d]_{q'}\ar[r]^-{\phi}& F(\mathbb{R}^{n+1},2)  \ar[d]_{q^{\mathbb{R}^{n+1}}}\ar@{^{(}->}[r]^{i} &  F(Y,2)\ar[d]^{q^{Y}} \\
X/\tau\ar[r]^{\overline{\rho}}&\mathbb{R}P^n\ar[r]^-{\overline{\phi}} &  D(\mathbb{R}^{n+1},2)\ar@{^{(}->}[r]^-{\overline{i} }    &  D(Y,2)  \\ 
       }
\end{eqnarray*} where $\phi(x)=(x,-x)$ and $i$ is induced by the inclusion $\mathbb{R}^{n+1}\hookrightarrow Y$. Then $\varphi:X\to F(Y,2)$ given by $\varphi=i\circ\phi\circ\rho$ is a $\mathbb{Z}_2$-equivariant continuous map such that the following diagram commutes
\begin{eqnarray*}
\xymatrix@C=2cm{ 
       X  \ar[r]^{\varphi}\ar[d]_{q} &  F(Y,2)\ar[d]^{q^Y} &\\
      X/\tau \ar[r]_{\overline{\varphi}} & D(Y,2) &
       }
\end{eqnarray*} and we conclude that $((X,\tau);Y)$ does not satisfy the BUP. 
\end{proof}

The following statement presents estimates of the sectional category $\mathrm{secat}(q^Y)$.
\begin{proposition}
If $Y$ is a connected topological manifold (without boundary) with dimension $n$ ($n\geq 1$), then
\begin{enumerate}
    \item $n\leq\mathrm{secat}(q^Y)\leq 2n+1$.
    \item If $\mathrm{hdim}(D(Y,2))\leq 2\dim(Y)-1$ then, $n\leq\mathrm{secat}(q^Y)\leq 2n$.
\end{enumerate}
\end{proposition}
\begin{proof} The case, $n= 1$ is shown easily. We will suppose that $n\geq 2$.  Let $\mathbb{R}^{n}\subset Y$ be an embedding. We have the following pullback:
\begin{eqnarray*}
\xymatrix{F(\mathbb{R}^{n},2)  \ar[d]_{q^{\mathbb{R}^{n}}}\ar@{^{(}->}[r]_{ } &  F(Y,2)\ar[d]^{q^{Y}} \\
D(\mathbb{R}^{n},2)\ar@{^{(}->}[r]_-{ }    &  D(Y,2)  \\
       }
\end{eqnarray*} Then \begin{align*}
    \mathrm{secat}(q^Y)&\geq \mathrm{secat}(q^{\mathbb{R}^{n}})\\
    &=n.
\end{align*} On the other hand,
\begin{enumerate}
    \item by Lemma~\ref{cat-stimates}, we have \begin{align*}
   \mathrm{secat}(q^Y)&\leq \mathrm{cat}\left(D(Y,2)\right)\\
   &\leq \mathrm{hdim}\left(D(Y,2)\right)+1\\
   &\leq 2n+1.\\
\end{align*}
\item By Lemma~\ref{cat-stimates} together with the hypotheses $\mathrm{hdim}(D(Y,2))\leq 2\dim(Y)-1$, we have \begin{align*}
   \mathrm{secat}(q^Y)&\leq \mathrm{cat}\left(D(Y,2)\right)\\
   &\leq \mathrm{hdim}\left(D(Y,2)\right)+1\\
   &\leq 2n-1+1\\
    &\leq 2n.\\
\end{align*}
\end{enumerate}
\end{proof}

Note that, any double cover with path-connected total space has sectional category at least $2$. Furthermore, for any connected $m$-dimensional CW complex $M^m$ and $\tau$ be a free cellular involution defined on $M^m$, the inequalities $2\leq \mathrm{secat}(q:M^m\to M^m/\tau)\leq m+1$ hold.

Now, we present a new lower bound for the index of $(M^m,\tau)$ in terms of the sectional category $\mathrm{secat}(q:M^m\to M^m/\tau)$.

\begin{theorem}\label{lower-bound-index} Let $X$ be a Hausdorff space admitting a fixed-point free involution $\tau$. Let $q:X\to X/\tau$ be the quotient map. 
If $1\leq n\leq \mathrm{secat}(q)-1$, then the triple $\left((X,\tau);\mathbb{R}^n\right)$ satisfies the BUP. In particular, the index of $\tau$ on $X$ is at least $\mathrm{secat}(q)-1$.
\end{theorem}
\begin{proof}
By Theorem~\ref{theorem-1} we have that $\left((X,\tau);\mathbb{R}^{\mathrm{secat}(q)-1}\right)$ satisfies the BUP.
\end{proof}

This lower bound can be achieved. Example~\ref{esfere-bup} shows that the index of the antipodal involution $A$ on $S^m$ is $m=\mathrm{secat}(q)-1$. More general, we have the following statement.

\begin{corollary}\label{m+1-m}
If $\mathrm{secat}(q:M^m\to M^m/\tau)=m+1$ then the index of $\tau$ on $M^{m}$ is $m$.
\end{corollary}
\begin{proof}
It follows from Theorem~\ref{lower-bound-index} together with the fact that the index is at most $m$ (see \cite[Lemma 2.4]{goncalves2010} or Lemma~\ref{n-n+1}).
\end{proof}

The \textit{category} of a map $f:X\to Y$, denote $\text{cat}(f)$, is the least integer $m$ such that $X$ can be covered by $m$ open sets $U_1,\ldots,U_m$, such that each restriction $f\mid_{U_i}$ is nullhomotopic. Note that, $\text{cat}(1_X)=\text{cat}(X)$. This numerical invariant was introduced by Berstein and Ganea in \cite{berstein1961}.

We recall basic properties concerning category of a map, see \cite{berstein1961}.

\begin{proposition}\label{bounds-cat-map}
\noindent
\begin{enumerate}
    \item If $f\simeq g$ then $\mathrm{cat}(f)=\mathrm{cat}(g)$.
    \item For any map $f:X\to Y$, we have $\mathrm{cat}(f)\leq\min\{\mathrm{cat}(X),\mathrm{cat}(Y)\}$.
    \item  We have $$\mathrm{cat}(f)\geq \mathrm{Nil}\left(\text{ im } (f^\ast)\right),$$ where $f^\ast:\widetilde{h}^\ast(Y)\to \widetilde{h}^\ast(X)$ denotes the induced homomorphism in any multiplicative reduced cohomology theory.
\end{enumerate}
\end{proposition}

We recall from \cite[Proposition 9.18, pg. 261]{cornea2003lusternik} how sectional category relates to the category of classifying maps.

\begin{proposition}\label{secat-category-classi}
Suppose $p:E\to B$ is a fibration arising as a pullback of a fibration $\hat{p}:\hat{E}\to \hat{B}$ 
\begin{eqnarray*}
\xymatrix{ E \ar[r]^{\tilde{f}} \ar[d]_{p} & \hat{E} \ar[d]^{\hat{p}} & \\
       B  \ar[r]_{f} &  \hat{B} &}
\end{eqnarray*}
where $\hat{E}$ is contractible. Then $\mathrm{secat}(p)=\mathrm{cat}(f)$.
\end{proposition}

Let $M^m$ be a connected $m$-dimensional CW complex and $\tau$ be a free cellular involution defined on $M^m$. We recall that for any principal $\mathbb{Z}_2$-bundle $q:M^m\to M^m/\tau$, we write $f_q:M^m/\tau\to\mathbb{R}P^\infty=B\mathbb{Z}_2$ for the classifying map of the bundle $q$. It is unique up to homotopy. Denoting the generator of $H^1(\mathbb{R}P^\infty;\mathbb{Z}_2)\cong\mathbb{Z}_2$ by $\alpha$, the characteristic class of the principal bundle is then $\gamma=f_q^\ast(\alpha)\in H^1(M^m/\tau;\mathbb{Z}_2)$. Since the bundle is non-trivial, it follows that $\gamma\neq 0$. In addition, by Proposition~\ref{secat-category-classi}, the sectional category $\mathrm{secat}(q:M^m\to M^m/\tau)=\mathrm{cat}(f_q)$. So, we obtain the following statement.

\begin{proposition}\label{characteristic-class}
Let $\gamma$ be the characteristic class of the principal bundle $q:M^m\to M^m/\tau$. For $n\leq m$, if $\gamma^n\neq 0$ then the triple $\left((M^m,\tau);\mathbb{R}^n\right)$ satisfies the BUP.
\end{proposition}
\begin{proof}
From Proposition~\ref{secat-category-classi}, $\mathrm{secat}(q)=\mathrm{cat}(f_q)$ where $f_q$ is the classifying map of $q$. By Proposition~\ref{bounds-cat-map}-item (3), $\mathrm{cat}(f_q)\geq n+1$ since $\gamma\in \text{ im } (f_q^\ast)$  and $\gamma^n\neq 0$. Then $\mathrm{secat}(q)\geq n+1$ and thus by Theorem~\ref{lower-bound-index}, we conclude that the triple $\left((M^m,\tau);\mathbb{R}^n\right)$ satisfies the BUP.
\end{proof}

The following implication of Proposition~\ref{characteristic-class} was proved in \cite[Theorem 3.4]{goncalves2010} for $n=m$. 

\begin{lemma}\cite[Theorem 3.4]{goncalves2010}\label{daci}
Let $\gamma$ be the characteristic class of the principal bundle $q:M^m\to M^m/\tau$.  The triple $\left((M^m,\tau);\mathbb{R}^m\right)$ satisfies the BUP if and only if $\gamma^m\neq 0$.
\end{lemma}

The following statement shows that the lower bound of the index given in Proposition~\ref{lower-bound-index} is also achieved. 

\begin{proposition}\label{m-1}
If $\mathrm{secat}(q:M^m\to M^m/\tau)=m$, then the index of $\tau$ on $M^m$ is $m-1$.
\end{proposition}
\begin{proof}
From Theorem~\ref{theorem-1}, we have the triple $\left((M^m,\tau);\mathbb{R}^n\right)$ satisfies the BUP for any $n<m$. Note that, $\gamma^m=0$ since $\mathrm{secat}(q:M^m\to M^m/\tau)=m$. Then, by Lemma~\ref{daci}, we have that the triple $\left((M^m,\tau);\mathbb{R}^m\right)$ does not satisfy the BUP, then the index of $\tau$ on $M^m$ is $m-1$. 
\end{proof}

From \cite[Proposition 1.27-(2), pg. 14]{cornea2003lusternik} we recall that $\mathrm{cat}(X\vee Y)=\max\{\mathrm{cat}(X),\mathrm{cat}(Y)\}$. Then, we have the following example which satisfies the condition of Proposition~\ref{m-1}.
\begin{example}
Let $M=S^{m}\vee S^{m-1}\vee S^{m}$ with $m\geq 3$ and $\tau$ be a free  cellular involution on $M$ such that $M/\tau=\mathbb{R}P^{m-1}\vee S^m$. Indeed, $\tau$ interchanges the two $S^m$ summands
from the wedge sums and acts antipodally on the $S^{m-1}$ summand. Similarly, like the calculation of the sectional category $\mathrm{secat}(S^{m-1}\to \mathbb{R}P^{m-1})=m$, we have that $\mathrm{secat}(q:S^{m}\vee S^{m-1}\vee S^{m}\to \mathbb{R}P^{m-1}\vee S^m)=\mathrm{cat}(\mathbb{R}P^{m-1}\vee S^m)=\mathrm{nil}\left(\text{Ker}(q^\ast_{\mathbb{Z}_2})\right)=m$. 
\end{example}

From Proposition~\ref{secat-produc}, the equality $\mathrm{secat}(p\times 1_Z)=\mathrm{secat}(p)$ holds for any fibration. Thus, we have the following example.

\begin{example}\label{example-final}
For any positive integers $m$ and $k$ such that $2\leq k\leq m+1$ consider the $m$-dimensional smooth manifold $M^m=S^{k-1}\times S^1\times\cdots\times S^1$ (product of one $S^{k-1}$ and $m-k+1$ copies of $S^1$) equipped with the free involution $\tau=A\times 1_{S^1}\times\cdots\times 1_{S^1}$ ($A$ the antipodal involution on $S^{k-1}$). Note that, the quotient map $q':M^m\to M^m/\tau$ coincidences with the product $q\times 1_{S^1}\times\cdots\times 1_{S^1}$, where $q:S^{k-1}\to \mathbb{R}P^{k-1}$ is the usual 2-covering map, and so, by Proposition~\ref{secat-produc}, we obtain that $\mathrm{secat}(q')=\mathrm{secat}(q)=k$. On the other hand, we have the following commutative diagram 
\begin{eqnarray*}
\xymatrix@C=3cm{ 
       S^{k-1}\times S^1\times\cdots\times S^1  \ar[r]^{\varphi}\ar[d]_{q\times 1_{S^1}\times\cdots\times 1_{S^1}} &  F(\mathbb{R}^k,2)\ar[d]^{q^{\mathbb{R}^k}} &\\
      \mathbb{R}P^{k-1}\times S^1\times\cdots\times S^1 \ar[r]_{\overline{\varphi}} & D(\mathbb{R}^k,2) &
       }
\end{eqnarray*} where $\varphi(x,z_1,\ldots,z_{n-k+1})=(x,-x)$ for any $(x,z_1,\ldots,z_{n-k+1})\in S^{k-1}\times S^1\times\cdots\times S^1$, and thus the triple $\left(\left(S^{k-1}\times S^1\times\cdots\times S^1,A\times 1_{S^1}\times\cdots\times 1_{S^1}\right);\mathbb{R}^k\right)$ does not satisfy the BUP. Therefore, by Proposition~\ref{lower-bound-index}, the index of $\left(M^m,\tau\right)$ is equal to $k-1=\mathrm{secat}(q')-1$ (compare with \cite[pg. 772]{dacibergpaiao2019}).
\end{example}

Motivated by Theorem~\ref{lower-bound-index}, Corollary~\ref{m+1-m}, Proposition~\ref{m-1}, and Example~\ref{example-final} we formulate the following statement. 

\begin{theorem}\label{thm:sec-index-one}
  Let $X$ be a Hausdorff paracompact space admitting a fixed-point free involution $\tau$. Let $q:X\to X/\tau$ be the quotient map. Then the index of $\tau$ on $X$ is equal to $\mathrm{secat}(q)-1$.  
\end{theorem}
\begin{proof}
Our proof make use of the converse implication from \cite[Theorem 9, pg. 86]{schwarz1966} for $G=\mathbb{Z}_2$.  In this case the principal fibre 
 space has total space the $r$-iterated join of  $\mathbb{Z}_2$, which is known to be the sphere $S^{r-1}$.
 
Let  $s=\mathrm{ind}_{\mathbb{Z}_2}(X)$ be the index of $\tau$ on $X$.  By Theorem~\ref{lower-bound-index}, we have that $s\geq \mathrm{secat}(q)-1$. By the converse implication from \cite[Theorem 9, pg. 86]{schwarz1966} (this is the only place where the hypothesis about the paracompactness of $X$ is used), since we do not have an equivariant map from $X$ into the  $S^{s-1}$ follows that $\mathrm{secat}(q)>s$ or   $\mathrm{secat}(q)-1\geq s$.
 So the index of $\tau$ on $X$ is equal to $\mathrm{secat}(q)-1$.   
\end{proof}

\begin{remark}
\noindent\begin{enumerate}
    \item[(1)] The direct  implication from \cite[Theorem 9, pg. 86]{schwarz1966} for $G=\mathbb{Z}_2$ (the implication which is not 
used in the proof above)  is a direct consequence of Theorem~\ref{lower-bound-index}.
\item[(2)] It turns out that Theorem~\ref{thm:sec-index-one} follows immediately from \cite[Theorem 30]{schwarz1966}, 
as we will see below, although the proof is  different from
the one given above which is more in the spirit of the present work.
 Namely,  let $T$ be a free involution on a Hausdorff paracompact $X$. We observe that $g(X,T)$ is the  sectional category 
 of the projection $X\to X/T$, which we denote by  $\mathrm{secat}(q)$. 
 Call $s=\mathrm{ind}_{\mathbb{Z}_2}(X)$ the index of the involution. So $b)$  of \cite[Theorem 30]{schwarz1966} holds for $n=s$, which  implies by item a)  that $\mathrm{secat}(q)>2n-n=n$. But b) does not hold for $n=s+1$ (from the definition of the index) so by the opposite of 
 $a)$ we have $\mathrm{secat}(q)\leq (n+1)2-n-1=n+1$. So follows that $\mathrm{secat}(q:X\to X/T)=s+1$. 
\end{enumerate}
\end{remark}

As a very simple application  of Theorem~\ref{thm:sec-index-one}, we determine  the sectional category of a double covering $S_1 \to S_2$ between any two closed surfaces. This is a corollary of  \cite[Theorem 5.5]{daciberg2006} together with the equality $\mathrm{secat}(q)=\mathrm{ind}_{\mathbb{Z}_2}(X)+1$.
 
 \begin{corollary}\label{cor} Let $(S, \tau )$ be a pair where $S$ is a closed surface and $\tau $ is a free $Z_2$
action. The sectional category of the projection map $q: S\to S/\tau$ is three if one of the following conditions below holds:
\begin{enumerate}
    \item[1)]  If $S$ is orientable and its Euler characteristic is congruent to 2 mod 4.
     \item[2)] $S$ is nonorientable, the Euler characteristic is congruent to 2 mod 4, and the
action $\tau$ is equivalent to one of the canonical actions (which correspond to the
subgroups given by the sequences of the form $(1, \delta_2, \delta_3, \ldots, \delta_{2r+1})$,  
where $\delta_i$ is either $\bar 1$ or  $\bar 0$).
\item[3)] S is nonorientable, the Euler characteristic is congruent to $0$ mod $4$ and $\tau$ is
equivalent to one of the canonical actions (which correspond to the subgroups
given by the sequences of the form $(1, \delta_2, \delta_3, \ldots, \delta_{2r})$, where $\delta_i$ is either $\bar 1$ or
 $\bar  0$).
\end{enumerate}
 Otherwise  it is two.
   \end{corollary} 

   \begin{remark}
      The particular case where $S$ is orientable and the quotient is nonorientable, from  Corollary~\ref{cor} we obtain: When $S/\tau$ has
 Euler characteristic odd ( i.e. a connected  sum of an odd numbers of projective planes) the sectional category is 
 $3$ and when $S/\tau$ has Euler characteristic even( i.e. a connected  sum of an even  numbers of projective planes) the sectional category is two. 
   \end{remark}


\section*{Conflict of Interest Statement}
On behalf of all authors, the corresponding author states that there is no conflict of interest.

\bibliographystyle{plain}

\end{document}